\documentclass[10pt]{amsart}
\usepackage{amssymb}
\usepackage{amscd}
\usepackage{amsmath}
\theoremstyle{plain}
\newtheorem{Thm}{Theorem}[section]
\newtheorem{Thm*}{Theorem}[section]
\newtheorem{Thm'}[Thm]{"Theorem"}
\newtheorem{Cor}[Thm]{Corollary}

\newtheorem{Prop}[Thm]{Proposition}
\newtheorem{Lem}[Thm]{Lemma}
\newtheorem{Cl}[Thm]{Claim}

\theoremstyle{definition}
\newtheorem{Def}[Thm]{Definition}

\newtheorem{Emp}[Thm]{}

\newtheorem{Q}[Thm]{Question}
\numberwithin{equation}{section}
\newcommand{\nc}{\newcommand}
\nc{\lm}{\lambda}
\newcommand{\pp}{\boxtimes}
\newcommand{\ov}{\overline}

\newcommand{\B}[1]{\mathbb#1}
\newcommand{\cal}[1]{\mathcal{#1}}
\newcommand{\C}[1]{\cal#1}

\newcommand{\isom}{\overset {\thicksim}{\to}}

\newcommand{\lra}{\longrightarrow}
\newcommand{\lla}{\longleftarrow}
\newcommand{\hra}{\hookrightarrow}
\newcommand{\wt}{\widetilde}
\newcommand{\Gm}{\Gamma}
\newcommand{\lan}{\langle}

\newcommand{\ran}{\rangle}

\newcommand{\Dt}{\Delta}
\newcommand{\La}{\Lambda}

\newcommand{\al}{\alpha}

\newcommand{\form}[1]{(\ref{Eq:#1})}

\newcommand{\rs}[1]{Section \ref{S:#1}}
\newcommand{\rl}[1]{Lemma \ref{L:#1}}

\newcommand{\rcl}[1]{Claim \ref{C:#1}}
\newcommand{\rp}[1]{Proposition \ref{P:#1}}

\newcommand{\re}[1]{\ref{E:#1}}
\newcommand{\rco}[1]{Corollary \ref{C:#1}}

\newcommand{\rt}[1] {Theorem \ref{T:#1}}

\newcommand{\rd}[1]{Definition \ref{D:#1}}
\newcommand{\sm}{\smallsetminus}

\newcommand{\pr}{\operatorname{pr}}

\newcommand{\Spec}{\operatorname{Spec}}

\newcommand{\Aut}{\operatorname{Aut}}

\newcommand{\Gr}{\operatorname{Gr}}

\newcommand{\Fix}{\operatorname{Fix}}

\newcommand{\Tr}{\operatorname{Tr}}

\newcommand{\red}{\operatorname{red}}

\newcommand{\Id}{\operatorname{Id}}

\newcommand{\Hom}{\operatorname{Hom}}

\newcommand{\pt}{\operatorname{pt}}

\newcommand{\ql}{\B{Q}_{\ell}}

\newcommand{\qlbar}{\overline{\ql}}

\newcommand{\Sp}{\operatorname{Sp}}

\begin{document}

\title[Local terms for transversal intersections]
{Local terms for transversal intersections}

\date{\today}

\author{Yakov Varshavsky}
\address{Institute of Mathematics\\
The Hebrew University of Jerusalem\\
Givat-Ram, Jerusalem,  91904\\
Israel} \email{vyakov@math.huji.ac.il}

\thanks{This research was partially supported by the ISF grant 822/17}
%\abstract
\begin{abstract}
The goal of this note is to show that in the case of ``transversal intersections'' the ``true local terms" appearing in the Lefschetz trace formula are equal to the ``naive local terms". To prove the result, we extend the method of \cite{Va}, where the case of contracting correspondences is treated. Our new ingredients are the observation of Verdier \cite{Ve} that specialization of an \'etale sheaf to the normal cone is monodromic and the assertion that local terms are ``constant in families". As an application, we get a generalization of the  Deligne--Lusztig trace formula \cite{DL}.
\end{abstract}
\maketitle

\centerline{\em Dedicated to Luc Illusie on the occasion of his 80th birthday}
%\tableofcontents
\section*{Introduction}

Let $f:X\to X$ be a morphism of schemes of finite type over an algebraically closed field $k$, let $\ell$ be a prime number different from the characteristic of
$k$, and let $\C{F}\in D_c^b(X,\qlbar)$ be equipped with a morphism $u:f^*\C{F}\to\C{F}$. Then for every fixed point $x\in \Fix(f)\subseteq X$, one can consider the
restriction $u_x:\C{F}_x\to\C{F}_x$, hence we can consider its trace $\Tr(u_x)\in\qlbar$, called the ``naive local term" of $u$ at $x$.

On the other hand, if $x\in\Fix(f)\subseteq X$ is an isolated fixed point, one can also consider the ``true local term" $LT_x(u)\in\qlbar$, appearing in the Lefschetz--Verdier trace formula,
so the natural question is when these two locals terms are equal.

Motivated by work of many people, including Illusie \cite{Il}, Pink \cite{Pi} and Fujiwara \cite{Fu}, it was shown in \cite{Va} that this is the case when $f$ is ``contracting near $x$",
by which we mean that the induced map of normal cones $N_x(f):N_x(X)\to N_x(X)$ maps $N_x(X)$ to the zero section. In particular, this happens when the induced map of Zariski tangent spaces $d_x(f):T_x(X)\to T_x(X)$ is zero.

A natural question is whether the equality  $LT_x(u)=\Tr(u_x)$ holds for a more general class of morphisms. For example, Deligne asked whether the equality holds when $x$ is the only fixed point of $d_x(f):T_x(X)\to T_x(X)$, or equivalently, when the linear map $d_x(f)-\Id:T_x(X)\to T_x(X)$ is invertible. Note that when $X$ is smooth at $x$, this condition is equivalent to the fact that  the graph of $f$ intersects transversally with the diagonal at $x$.

The main result of this note gives an affirmative answer to Deligne's question. Moreover, in order to get an equality $LT_x(u)=\Tr(u_x)$ it suffices to assume a weaker condition that $x$ is the only fixed point of $N_x(f):N_x(X)\to N_x(X)$ (see \rco{main}). In particular, we show this  in the case when $f$ is an automorphism of $X$ of finite order, prime to the characteristic of $k$, or, more generally, a ``semisimple" automorphism (see \rco{eqlocterms}).

Actually, as in \cite{Va}, we show a more general result (see \rt{main}) in which a morphism $f$ is replaced by a correspondence, and a fixed point $x$ is replaced by a $c$-invariant closed subscheme $Z\subseteq X$. Moreover, instead of showing the equality of local terms we show a more general ``local" assertion that in some cases the so-called ``trace maps" commute with restrictions. Namely, we show it in the case when $c$ has ``no almost fixed points in the punctured tubular neighborhood of $Z$" (see \rd{good}).

As an easy application, we prove a generalization of the Deligne--Lusztig trace formula (see \rt{dl}).

To prove our result, we follow the strategy of \cite{Va}. First, using additivity of traces, we reduce to the case when $\C{F}_x\simeq 0$. In this case, $\Tr(u_x)=0$, thus we have to show that
$LT_x(u)=0$. Next, using specialization to the normal cone, we reduce to the case when $f:X\to X$ is replaced by $N_x(f):N_x(X)\to N_x(X)$ and $\C{F}$ by its specialization $sp_x(\C{F})$.
In other words, we can assume that $X$ is a cone with vertex $x$, and $f$ is $\B{G}_m$-equivariant.

In the contracting case, treated in \cite{Va}, the argument stops there. Indeed, after passing to normal cones we can assume that $f$ is the constant map with image $x$. In this case, our assumption $\C{F}_x\simeq 0$ implies that  $f^*\C{F}\simeq 0$, thus $u=0$, hence $LT_x(u)=0$.

In general, by a theorem of Verdier \cite{Ve}, we can assume that $\C{F}$ is monodromic. Since it is enough to show an analogous  assertion for sheaves with finite coefficients, we can thus assume that $\C{F}$ is $\B{G}_m$-equivariant
with respect to the action $(t,y)\mapsto t^n(y)$ for some $n$.

Since $f$ is homotopic to the constant map with image $x$ (via the homotopy $f_t(y):=t^n f(y)$) it suffices to show that local terms are ``constant in families". We deduce the latter assertion from the fact that local terms commute with nearby cycles.

The paper is organized as follows. In Section 1 we introduce correspondences, trace maps and local terms.  In Section 2 we define relative correspondences and formulate Proposition 2.5 asserting that in some cases trace maps are ``constant in families". In Section 3 we study a particular case of relative correspondences, obtained from schemes with an action of an algebraic monoid $(\B{A}^1,\cdot)$.

In Section 4 we formulate our main result (\rt{main}) asserting that in some cases trace maps commute with restrictions to closed subschemes. We also deduce an affirmative answer to Deligne's question, discussed earlier. In Section 5 we apply the results of Section 4 to the case of an automorphism and deduce a generalization of the Deligne--Lusztig trace formula. 

Finally, we prove \rt{main} in Section 6 and prove Proposition 2.5 in Section 7. 

I thank Luc Illusie, who explained to me a question of Deligne several years ago and expressed his interest on many occasions. I also thank
David Hansen and Jared Weinstein for their comments and stimulating questions (see \re{KW}) and thank Helene Esnault and Nick Rozenblyum for their interest.

\section*{Notation}

For a scheme $X$, we denote by $X_{\red}$ the corresponding reduced scheme. For a morphism of schemes $f:Y\to X$ and a closed subscheme $Z\subseteq X$, we denote by $f^{-1}(Z)\subseteq Y$ the schematic inverse image of $Z$.

Throughout most of the paper all schemes will be of finite type over a fixed algebraically closed field $k$. The only exception is  \rs{rpconst}, where all schemes will be of finite type over a spectrum of a discrete valuation ring over $k$ with residue field $k$.

We fix a prime $\ell$,  invertible in $k$, and a commutative ring with identity $\La$, which is either finite and is annihilated by some
power of $\ell$, or a  finite extension of $\B{Z}_{\ell}$ or $\B{Q}_{\ell}$.

To each scheme $X$ as above, we associate a category $D_{ctf}^b(X,\La)$ of
``complexes of finite tor-dimension
with constructible cohomology'' (see \cite[Rapport 4.6]{SGA4.5} when $\La$ is finite and
\cite[1.1.2-3]{De} in other cases). This category is known to be stable under the six
operations  $f^*, f^!, f_*, f_!, \otimes$ and $\C{RHom}$ (see \cite[Th. finitude, 1.7]{SGA4.5}). %In Section \ref{S:verdier} we will work with a
%larger category $D_{c}^b(X,\La)$ of ``complexes with constructible cohomology''.

For each $X$ as above, we denote by $\pi_X:X\to\pt:=\Spec k$ the structure morphism, by
$\La_X\in D_{ctf}^b(X,\La)$ the constant sheaf with fiber $\La$, and
by $K_X=\pi_X^!(\La_{\pt})$ the dualizing complex of $X$.
%by $\B{D}=\B{D}_X:=\C{RHom}(\cdot,K_X)$ the Verdier duality.
We will also write $R\Gm(X,\cdot)$ (resp. $R\Gm_c(X,\cdot)$) instead of $\pi_{X*}$ (resp $\pi_{X!}$).

For an embedding $i:Y\hra X$ and $\C{F}\in D_{ctf}^b(X,\La)$, we will often write $\C{F}|_Y$ instead of
$i^*\C{F}$.

We will freely use various base change morphisms
(see, for example, \cite[XVII, 2.1.3 and XVIII, 3.1.12.3, 3.1.13.2, 3.1.14.2]{SGA4}), which we will denote by $BC$.

\section{Correspondences and trace maps}

\begin{Emp} \label{E:correspond}
{\bf Correspondences.}
(a)  By a {\em correspondence}, we mean a morphism of schemes of the form
$c=(c_l,c_r):C\to X\times X$, which can be also viewed as a diagram
$X\overset{c_l}{\lla}C  \overset{c_r}{\lra}X$.

(b) Let $c:C\to X\times X$ and $b:B\to Y\times Y$ be correspondences.
By a {\em morphism} from $c$ to $b$, we mean a pair of morphisms $[f]=(f,g)$, making the following
diagram commutative
\begin{equation} \label{Eq:funct}
\CD
        X   @<{c_l}<<                    C        @>{c_r}>>         X\\
        @V{f}VV                        @V{g}VV                       @VV{f}V\\
        Y @<{b_l}<<                    B    @>{b_r}>>            Y.
\endCD
\end{equation}

(c) A correspondence $c:C\to X\times X$ gives rise to a Cartesian diagram
\[
\CD
\Fix(c) @>>> C\\
@VVV @V{c}VV\\
X@>\Dt>> X\times X,
\endCD
\]
where $\Dt:X\to X\times X$ is the diagonal map. We call $\Fix(c)$ {\em the scheme of fixed points of $c$}.
%Since $\Dt$ is a locally closed embedding, so is $\Dt'$.

(d) We call a morphism $[f]$ from (b) {\em Cartesian}, if the right inner square of \form{funct} is Cartesian.
% up to nilpotents", that is, the induced morphism $C_{\red}\to (B\times_Y X)_{\red}$ is an isomorphism.
\end{Emp}

\begin{Emp} \label{E:restr}
{\bf Restriction of correspondences.}
Let $c:C\to X\times X$ be a correspondence,  $W\subseteq C$ an open subscheme, and $Z\subseteq X$ a locally closed subscheme.

(a) We denote by $c|_W:W\to X\times X$ the restriction of $c$.

(b) Let $c|_Z:c^{-1}(Z\times Z)\to Z\times Z$ be the restriction of $c$. By definition, the inclusion maps
$Z\hra X$ and $c^{-1}(Z\times Z)\hra C$ define a morphism $c|_Z\to c$ of correspondences.

%we have a commutative diagram
%
%\[
%\CD
%        Z   @<(c|_Z)_l<<            c^{-1}(Z\times Z)      @>(c|_Z)_r>>         Z\\
%        @VVV                        @VVV                       @VVV\\
%        X @<{c_l}<<                    C    @>{c_r}>>            X.
%\endCD
%\]
%

(c) We say that $Z$ is {\em (schematically) $c$-invariant}, if $c_r^{-1}(Z)\subseteq c_l^{-1}(Z)$. This happens if and only if we have
$c^{-1}(Z\times Z)=c_r^{-1}(Z)$ or, equivalently, 
the natural morphism of correspondences $c|_Z\to c$ from (b) is Cartesian.
%the right inner square of the diagram in (b) is Cartesian up to nilpotents.

\end{Emp}

\begin{Emp} \label{E:rem}
{\bf Remark.} Our conventions slightly differ from those of \cite[1.5.6]{Va}. For example, we do not assume that $Z$ is closed, 
our notion of $c$-invariance is stronger than the one of \cite[1.5.1]{Va}, and when $Z$ is $c$-invariant, then
$c|_Z$ in the sense of \cite{Va} is the correspondence $c^{-1}(Z\times Z)_{\red}\to Z\times Z$.
\end{Emp}

\begin{Emp} \label{E:cohcorr}
{\bf Cohomological correspondences.}
Let $c:C\to X\times X$ be a correspondence, and let $\C{F}\in D^b_{ctf}(X,\La)$.

(a) By {\em $c$-morphism} or a {\em cohomological correspondence lifting $c$}, we mean an element of
\[
\Hom_c(\C{F},\C{F}):=\Hom(c_l^*\C{F},c_r^!\C{F})\simeq\Hom(c_{r!}c_l^*\C{F},\C{F}).
\]

(b) Let $[f]:c\to b$ be a Cartesian morphism of correspondences (see \re{correspond}(d)).
Then every $b$-morphism $u:b_l^*\C{F}\to b_r^!\C{F}$ gives rise to a $c$-morphism $[f]^*(u):c_l^*(f^*\C{F})\to c_r^!(f^*\C{F})$ 
defined as a composition
\[
c_l^*(f^*\C{F})\simeq g^*(b_l^*\C{F})\overset{u}{\lra} g^*(b_r^!\C{F})\overset{BC}{\lra}c_r^!(f^*\C{F}),
\]
where base change morphism $BC$ exists, because $[f]$ is Cartesian.

(c) As in \cite[1.1.9]{Va}, for an open subset $W\subseteq C$, every $c$-morphism $u$ gives rise to a $c|_W$-morphism $u|_W:(c_l^*\C{F})|_W\to (c_r^!\C{F})|_W$.

(d) It follows from (b) and \re{restr}(c) that for a $c$-invariant subscheme $Z\subseteq X$, every $c$-morphism $u$ gives rise to a $c|_Z$-morphism
$u|_Z$ (compare \cite[1.5.6(a)]{Va}).
%(d) For a general $Z$,  we can consider the $c|^0_Z$-morphism $(u|_{W(Z)})|_Z$ (see (a),(b) and \re{restr}, which we will simply denote by
%$u|_Z$ (compare\cite[1.5.6(b)]{Va}).
\end{Emp}

%(b)  We say that $Z$ is {\em $c$-invariant in a neighborhood of fixed points},
%if there exists an open neighborhood $W\subseteq C$ of $\Dt'(\Fix(c))$ such that
%$Z$ is $c|_W$-invariant. Equivalently, this happens if and only if the open subset $W=W(Z)$ of \cite[Lem 1.5.3](b)]{Va}) contains
%$\Dt'(\Fix(c))$.

\begin{Emp} \label{E:locterms}
{\bf Trace maps and local terms.} Fix  a correspondence $c:C\to X\times X$.

(a) As in \cite[1.2.2]{Va}, to every $\C{F}\in D^b_{ctf}(X,\La)$ we associate the trace map

\[\C{Tr}_c:\Hom_c(\C{F},\C{F})\to H^0(\Fix(c),K_{\Fix(c)}).\]

(b)  For an {\em open subset} $\beta$ of $\Fix(c)$,\footnote{by which we mean that $\beta\subseteq\Fix(c)$ is a locally closed subscheme such that
$\beta_{\red}\subseteq\Fix(c)_{\red}$ is open} we denote by
\begin{equation*} \label{Eq:lterm2}
\C{Tr}_{\beta}: \Hom_c(\C{F},\C{F})\to H^0(\beta, K_{\beta})
\end{equation*}
the composition of $\C{Tr}_c$ and the restriction map
$H^0(\Fix(c),K_{\Fix(c)})\to  H^0(\beta, K_{\beta})$.

(c) If in addition $\beta$ is proper over $k$, we denote by
\begin{equation*} \label{Eq:lterm3}
LT_{\beta}:\Hom_c(\C{F},\C{F})\to\La
\end{equation*}
the composition of $\C{Tr}_{\beta}$ and the integration map
$\pi_{\beta!}:H^0(\beta, K_{\beta})\to \La$.

(d) In the case when $\beta$ is a {\em connected component} of $\Fix(c)$,\footnote{that is, $\beta$ is a closed connected subscheme of $\Fix(c)$ such that $\beta_{\red}\subseteq \Fix(c)_{\red}$ is open} which is proper
over $k$, $LT_{\beta}(u)$ is usually
called the {\em (true) local term} of $u$ at $\beta$.
\end{Emp}

\section{Relative correspondences}

\begin{Emp} \label{E:relcor}
{\bf Relative correspondences.} Let $S$ be a scheme over $k$. By a {\em relative correspondences} over $S$, we mean a morphism ${c}=({c}_l,{c}_r):{C}\to {X}\times_S{X}$ of schemes over $S$, or equivalently, a correspondence ${c}=({c}_l,{c}_r):{C}\to{X}\times{X}$ such that ${c}_l$ and ${c}_r$ are morphisms over $S$.

(a) For a correspondence $c$ as above and a morphism $g:S'\to S$ of schemes over $k$ we can form a relative correspondence $g^*({c}):={c}\times_S S'$ over $S'$. Moreover, it follows from \re{cohcorr}(b) that every $c$-morphism ${u}\in\Hom_c(\C{F},\C{F})$ gives rise to the
$g^*({c})$-morphism $g^*({u})\in \Hom_{g^*(c)}(g^*\C{F},g^*\C{F})$, where $g^*\C{F}\in D^b_{ctf}(X\times_S S',\La)$ denotes the $*$-pullback of $\C{F}$.   

(b) For a geometric point $s$ of $S$, let $i_s:\{s\}\to S$ be the canonical map, and we set ${c}_s:=i_s^*({c})$.
Then, by (a), every ${c}$-morphism ${u}\in \Hom_{{c}}(\C{F},\C{F})$ gives rise to a ${c}_s$-morphism  ${u}_s:=i_s^*({u})\in\Hom_{{c}_s}(\C{F}_s,\C{F}_s)$.
Thus we can form the trace map $\C{Tr}_{{c}_s}({u}_{s})\in H^0(\Fix({c}_s), K_{\Fix({c}_s)})$.
\end{Emp}

\begin{Emp}
{\bf Remark.}
In other words, a relative correspondence  ${c}$ over $S$ gives rise a family of correspondences ${c}_s:{C}_s\to {X}_s\times {X}_s$, parameterized by a collection of geometric points
$s$ of $S$. Moreover, every ${c}$-morphism
${u}$ gives rise to a family of ${c}_s$-morphisms ${u}_s\in\Hom_{{c}_s}(\C{F}_s,\C{F}_s)$, thus a family of trace maps
$\C{Tr}_{{c}_s}({u}_s)\in H^0(\Fix({c}_s), K_{\Fix({c}_s}))$.
\end{Emp}

\rp{const} below, whose proof will be given in \rs{rpconst},
 asserts that in some cases the assignment $s\mapsto \C{Tr}_{{c}_s}({u}_s)$ is ``constant".

\begin{Emp} \label{E:family}
{\bf Notation.} We say that a morphism  $f:X\to S$ is a {\em topologically constant family}, if  the reduced scheme ${X}_{\red}$ is isomorphic to a product
$Y\times S_{\red}$ over $S$.
\end{Emp}

\begin{Cl} \label{C:family}
Assume that  $f:X\to S$ is a topologically constant family, and that $S$ is connected. Then for every two geometric points $s,t$ of $S$, we have a canonical identification
$R\Gm({X}_{s},K_{{X}_{s}})\simeq R\Gm({X}_{t},K_{X_t})$, hence $H^0({X}_{s},K_{{X}_{s}})\simeq H^0({X}_{t},K_{X_t})$.
\end{Cl}

\begin{proof}
Set $K_{X/S}:=f^!(\La_S)\in D_{ctf}^b(X,\La)$ and  $\C{F}:=f_*(K_{X/S})\in D_{ctf}^b(S,\La)$. Our assumption on $f$ implies that for every geometric point $s$ of $S$, the base change morphisms
\[
\C{F}_s=R\Gm(s,\C{F}_s)\to R\Gm({X}_{s},i_s^*(K_{X/S}))\to R\Gm({X}_{s},K_{{X}_{s}})
\]
are isomorphisms. Furthermore, the assumption also implies that $\C{F}$ is constant, that is, isomorphic to a pullback of an object in $D_{ctf}^b(\pt,\La)$. Then for every specialization arrow $\al:t\to s$, the specialization map $\al^*:\C{F}_s\to\C{F}_t$ (see \cite[VIII, 7]{SGA4}) is an isomorphism (because $\C{F}$ is locally constant), and does not depend on the specialization arrow $\al$ (only on $s$ and $t$). Thus the assertion follows from the assumption that $S$ is connected.
\end{proof}

%The proof of the following result will be given in Section \ref{S:const}.

\begin{Prop} \label{P:const}
Let ${c}:{C}\to {X}\times {X}$ be a relative correspondence over $S$ such that  $S$ is connected, and that $\Fix({c})\to S$ is a topologically constant family.

Then for every  ${c}$-morphism ${u}\in \Hom_{{c}}(\C{F},\C{F})$ such that  $\C{F}$ is ULA over $S$, the assignment  $s\mapsto \C{Tr}_{{c}_s}({u}_s)$ is ``constant", that is, for every two geometric points $s,t$ of $S$, the identification $H^0({X}_{s},K_{{X}_{s}})\simeq H^0({X}_{t},K_{{X}_{t}})$
from \rcl{family} identifies
$\C{Tr}_{{c}_{s}}({u}_{s})$ with  $\C{Tr}_{{c}_{t}}({u}_{t})$.

In particular, we have  $\C{Tr}_{{c}_{s}}({u}_{s})=0$ if and only if $\C{Tr}_{{c}_{t}}({u}_{t})=0$.
\end{Prop}

\section{An $(\B{A}^1,\cdot)$-equivariant case}

\begin{Emp} \label{E:cons}
{\bf Construction.} Fix a scheme $S$ over $k$ and a morphism $\mu:X\times S\to X$.

(a) A correspondence $c:C\to X\times X$ gives rise to the correspondence $c_S={c}_S^{\mu}:C_S\to X_S\times_{S} X_S$ over $S$, where $C_S:=C\times S$ and $X_S:=X\times S$, while $c_{Sl},c_{Sr}:C\times S\to X\times S$ are given by \[c_{Sr}:=c_r\times\Id_S \text{ and }  c_{Sl}:=(\mu,\pr_S)\circ(c_l\times\Id_S),\] that is, $c_{Sl}(y,s)= (\mu(c_l(y),s),s)$ and $c_{Sr}(y,s)=(c_r(y),s)$ for all $y\in C$ and $s\in S$.

(b) For every geometric point $s$ of $S$, we get an endomorphism $\mu_s:=\mu(-,s):X_s\to X_s$. Then $c_s:=i_s^*(c_S)$ is the correspondence
$c_s=(\mu_s\circ c_l,c_r):C_s\to X_s\times X_s$. In particular, for every $s\in S(k)$ we get a correspondence $c_s:C\to X\times X$.

(c) Suppose we are give $\C{F}\in D_{ctf}^b(X,\La)$, a $c$-morphism $u\in\Hom_c(\C{F},\C{F})$ and a morphism $v:\mu^*\C{F}\to\C{F}_S$ in $D_{ctf}^b(X_S,\La)$, where we set $\C{F}_S:=\C{F}\pp\La_S\in D_{ctf}^b(X_S,\La)$. To this data we associate a $c_S$-morphism $u_S\in\Hom_{c_S}(\C{F}_S,\C{F}_S)$, defined as a composition

\[
c_{Sl}^*(\C{F}_S)\simeq (c_l\times\Id_S)^*(\mu^*\C{F})\overset{v}{\to}(c_l\times\Id_S)^*(\C{F}_S)\simeq (c_l^*\C{F})\pp\La_S\overset{u}{\to} (c_r^!\C{F})\pp\La_S\simeq
c_{Sr}^!(\C{F}_S).
\]

(d) For every geometric point $s$ of $S$, morphism $v$ restricts to a morphism $v_s=i_s^*(v):\mu_s^*\C{F}\to\C{F}$, and the $c_s$-morphism
$u_s:=i_s^*(u_S):c_l^*\mu_s^*\C{F}\to c_r^!\C{F}$ decomposes as
\[u_s:c_l^*\mu_s^*\C{F}\overset{v_s}{\lra}c_l^*\C{F}\overset{u}{\lra}c_r^!\C{F}.\]
\end{Emp}

\begin{Emp} \label{E:examples}
{\bf Remarks.} For a morphism $\mu:X\times S\to X$ and a closed point $a\in S$, we set $S^a:=S\sm \{a\}$, and $\mu^a:=\mu|_{X\times S^a}:X\times S^a\to X$. Let $\C{F}\in D_{ctf}^b(X,\La)$ be such that $\mu_a^*\C{F}\simeq 0$. 

(a) Every morphism $v^a:(\mu^a)^*\C{F}\to \C{F}_{S^a}$ uniquely extend to a morphism $v:\mu^*\C{F}\to\C{F}_S$.

Indeed, let $j:X\times S^a\hra X\times S$ and $i:X\times \{a\}\hra X\times S$ be the inclusions.
 Using distinguished triangle $j_!j^*\mu^*\C{F}\to\mu^*\C{F}\to i_*i^*\mu^*\C{F}$ and the assumption that $i^*\mu^*\C{F}\simeq \mu_a^*\C{F}\simeq 0$, we conclude that the map $j_!j^*\mu^*\C{F}\to\mu^*\C{F}$ is an isomorphism.
Therefore the restriction map
\[j^*:\Hom(\mu^*\C{F},\C{F}_S)\to \Hom(j^*\mu^*\C{F},j^*\C{F}_S)\simeq \Hom(j_!j^*\mu^*\C{F},\C{F}_S)\] is an isomorphism, as claimed.

(b) Our assumption $\mu_a^*\C{F}\simeq 0$ implies that
$\Hom_{c_a}(\C{F},\C{F})=\Hom(c_l^*\mu_a^*\C{F},c_r^!\C{F})\simeq 0$. 
\end{Emp}

\begin{Emp} \label{E:equiv case}
{\bf Equivariant case.} (a) Assume that $S$ is an algebraic monoid, acting on $X$, and that $\mu:X\times S\to X$ is the action map. We say that $\C{F}\in D_{ctf}^b(X,\La)$ is {\em weakly $S$-equivariant}, if we are given
a morphism $v:\mu^*\C{F}\to\C{F}_S$ such that $v_1:\C{F}=\mu_1^*\C{F}\to \C{F}$ is the identity.
In particular, the construction of \re{cons} applies, so to every $c$-morphism $u\in\Hom_c(\C{F},\C{F})$ we associate a $c_S$-morphism $u_S\in\Hom_{c_S}(\C{F}_S,\C{F}_S)$.

(b) In the situation of (a), the correspondence $c_1$ equals $c$, and the assumption on $v_1$ implies that
the $c$-morphism $u_1$ equals $u$.
\end{Emp}

%\begin{Emp} \label{E:cones}
%{\bf $.} (a) Let $X$ be a scheme, equipped with an action $\mu:X\times \B{A}^1\to X$ of the algebraic semigroup $\B{A}^1$, and let $Z:=X^{\B{G}_m}$ be the scheme
%of $\B{G}_m$-fixed points. Then $\mu_0:X\to X$ factors as  $X\to Z\subseteq X$ and $\mu_0|_Z$ is the identity. In particular, $\mu_t(z)=z$ for every $t\in\B{A}^1$ and $z\in Z$.

%(a) Let $Z$ be a scheme. Recall that  a {\em cone over $Z$} is a scheme of the form
%$X=\C{Spec}(\C{A})$, where $\C{A}=\bigoplus_{n=0}^{\infty} \C{A}_n$ is a graded quasicoherent $\C{O}_Z$-algebra where $\C{A}_0=\C{O}_Z$ and
%each $\C{A}_n$ is a coherent $\C{O}_Z$-module.

%(b) Note that if $X$ is a cone over $Z$, then the projection $X\to Z$ has a natural section {\em the zero section} $Z\hra X$, identifying $Z$ with a closed subscheme of $X$.
%Moreover, $X$ is equipped with a natural action $\mu:X\times \B{A}^1\to X$ of the algebraic semigroup $\B{A}^1$ over $Z$ such that $Z\subseteq X$ is the scheme of $\B{G}_m$-fixed points, and
%$\mu_0:X\to X$ is the composition  $X\to Z\hra X$.

%(c) Notice that category of cones with $\B{G}_m$-equivariant morphisms admits fiber products. Namely, if $X_i$ is a cone over $Z_i$, and $X_1\to X_2\lla X_3$ is a $\B{G}_m$-equivariant diagram, then
%$X_1\times_{X_2} X_3$ is a cone over $Z_1\times_{Z_2} Z_3$.
%\end{Emp}

\begin{Emp} \label{E:famcorr}
{\bf Basic example.} (a) Let $X$ be a scheme, equipped with an action $\mu:X\times \B{A}^1\to X$ of the algebraic monoid $(\B{A}^1,\cdot)$, let $\mu_0:X\to X$ be the induced (idempotent) endomorphism, and let $Z=Z_X\subseteq X$ be the scheme of $\mu_0$-fixed points, also called the {\em zero section}. Then $Z_X\subseteq X$ is a locally closed subscheme, while $\mu_0:X\to X$ factors as $X\to Z_X\hra X$, thus inducing a projection $\pr_X:X\to Z_X$, whose restriction to $Z_X$ is the identity.

(b) The correspondence $X\mapsto (Z_X\subseteq X\overset{\pr_X}{\lra} Z_X)$ is functorial. Namely, every $(\B{A}^1,\cdot)$-equivariant morphism
$f:X'\to X$ induced a morphism $Z_f:Z_{X'}\to Z_X$ between zero sections, and we have an equality $Z_f\circ \pr_{X'}=\pr_X\circ f$ of morphisms 
$X'\to Z_X$.  

(c) Let $c:C\to X\times X$ be any correspondence. Then the construction of \re{cons} gives rise to a relative correspondence 
$c_{\B{A}^1}:C_{\B{A}^1}\to C_{\B{A}^1}\times C_{\B{A}^1}$ over $\B{A}^1$, hence a family of correspondences $c_t:C\to X\times X$, parameterized by $t\in\B{A}^1(k)$.

(d) For every $t\in\B{A}^1(k)$, the zero section $Z\subseteq X$ is $\mu_t$-invariant, and the induced map $\mu_t|_Z$ is the identity.
Therefore we have an inclusion $\Fix(c|_Z)\subseteq\Fix(c_t|_Z)$ of schemes of fixed points.

(e) For every $t\in\B{G}_m(k)$, we have an equality $\Fix(c_t|_Z)=\Fix(c|_Z)$. Indeed, one inclusion was shown in (d), while the opposite one  follows from first one together with identity  $(c_t)_{t^{-1}}=c$.

(f) Since $\mu_0$ factors through $Z\subseteq X$, we have an equality $\Fix(c_0|_Z)=\Fix(c_0)$.
Moreover, if $Z$ is $c$-invariant, we have an equality  $\Fix(c_0|_Z)=\Fix(c|_Z)$. Indeed, one inclusion was shown in (d), while the opposite one 
follows from the inclusion $\Fix(c_0|_Z)\subseteq c_r^{-1}(Z)=c^{-1}(Z\times Z)$.
\end{Emp}

\begin{Emp} \label{E:twisted}
{\bf Twisted action.} Assume that we are in the situation of \re{famcorr}. For every $n\in\B{N}$, we can consider the $n$-twisted action $\mu(n):X\times \B{A}^1\to X$ of $(\B{A}^1,\cdot)$ on $X$ given by formula $\mu(n)(x,t)=\mu(x,t^n)$. It gives rise to the family of correspondences $c^{\mu(n)}_{t}:C\to X\times X$ such that $c^{\mu(n)}_{t}=c_{t^n}$. Clearly, $\mu(n)$ restricts to an
$n$-twisted action of $\B{G}_m$ on $X$. 
\end{Emp}

\begin{Prop} \label{P:cones}
Let $X$ be an $(\B{A}^1,\cdot)$-equivariant scheme, and let $c:C\to X\times X$ be a correspondence such that $Z=Z_X\subseteq X$ is closed and $c$-invariant, $\Fix(c)\sm\Fix(c|_Z)=\emptyset$ and the set $\{t\in \B{A}^1(k)\,|\,\Fix(c^{\mu}_t)\sm\Fix(c^{\mu}_t|_Z)\neq\emptyset\}$ is finite. 

Then for every weakly $\B{G}_m$-equivariant $\C{F}\in D_{ctf}^b(X,\La)$ (see \re{equiv case}(a)) with respect to the $n$-twisted action (see \re{twisted}) such that $\C{F}|_Z=0$ and every $c$-morphism $u\in\Hom_c(\C{F},\C{F})$, we have $\C{Tr}_c(u)=0$.
\end{Prop}

\begin{proof}
Consider the $n$-twisted action $\mu(n):X\times \B{A}^1\to X$, and let $\mu(n)^0:X\times\B{G}_m\to X$ be the induced $n$-twisted action of $\B{G}_m$. The weakly $\B{G}_m$-equivariant structure on $\C{F}$ gives rise to the morphism $v^0:(\mu(n)^0)^*\C{F}\to \C{F}_{\B{G}_m}$ (see \re{equiv case}(a)).

Next, since $\mu(n)_0=\mu_0:X\to X$ factors through $Z$, while $\C{F}|_Z=0$, we conclude that $(\mu(n)_0)^*\C{F}\simeq 0$. Therefore morphism $v^0$ extends uniquely to the morphism
$v:\mu(n)^*\C{F}\to\C{F}_{\B{A}^1}$ (see \re{examples}(a)). Thus by construction \re{cons}(c), our $c$-morphism $u$ gives rise to the $c^{\mu(n)}_{\B{A}^1}$-morphism $u_{\B{A}^1}\in\Hom_{c^{\mu(n)}_{\B{A}^1}}(\C{F}_{\B{A}^1},\C{F}_{\B{A}^1})$
such that $u_1=u$ (see \re{equiv case}(b)).

Notice that since $u_0\in\Hom_{c_0}(\C{F},\C{F})=0$ (see \re{examples}(b)), we have $\C{Tr}_{c_0}(u_0)=0$. We would like to apply \rp{const} to deduce that $\C{Tr}_c(u)=\C{Tr}_{c_1}(u_1)=0$.

Consider the set $T:=\{t\in\B{A}^1(k)\,|\,\Fix(c^{\mu}_{t^n})\sm\Fix(c^{\mu}_{t^n}|_Z)\neq\emptyset\}$. Then $0\notin T$ (by \re{famcorr}(f)), and
our assumption says that $T$ is finite, and $1\notin T$. 
Then $S:=\B{A}^1\sm T\subseteq\B{A}^1$ is an open subscheme, and $0,1\in S$. Let $c^{\mu(n)}_S$ be the restriction of $c^{\mu(n)}_{\B{A}^1}$ to $S$, and it suffices to show that $\Fix(c^{\mu(n)}_S)\to S$ is a topologically constant family, thus \rp{const} applies.

We claim that we have the equality $\Fix(c^{\mu(n)}_S)_{\red}=\Fix(c^{\mu(n)}_S|_{Z\times S})_{\red}=\Fix(c|_Z)_{\red}\times S$ of locally
closed subschemes of $C\times S$. For this it suffices to show that for every $t\in S(k)$ we have equalities
$\Fix(c_{t^n})_{\red}=\Fix(c_{t^n}|_{Z})_{\red}=\Fix(c|_Z)_{\red}$. 
The left equality follows from the identity $\Fix(c^{\mu}_{t^n})\sm \Fix(c^{\mu}_{t^n}|_Z)=\emptyset$ used to define $S$. 
Since $Z$ is $c$-invariant, while the right equality follows from observations \re{famcorr}(e),(f).
\end{proof}

\begin{Emp} \label{E:equivcor}
{\bf Equivariant correspondences.} Let $c:C\to X\times X$ be an {\em $(\B{A}^1,\cdot)$-equivariant correspondence}, by which we mean that both
$C$ and $X$ are equipped with an action of a monoid $(\B{A}^1,\cdot)$, and both projections $c_l,c_r:C\to X$ are $(\B{A}^1,\cdot)$-equivariant.

(a) Note that the subscheme of fixed points $\Fix(c)\subseteq C$ is $(\B{A}^1,\cdot)$-invariant,
correspondence $c$ induces a correspondence $Z_c:Z_C\to Z_X\times Z_X$ between zero sections, and we have an equality $\Fix(Z_c)=Z_{\Fix(c)}$ 
of locally closed subschemes of $C$. 

(b) By \re{cons}(a), correspondence $c$ gives rise to a relative correspondence $c_{\B{A}^1}:C_{\B{A}^1}\to X_{\B{A}^1}\times X_{\B{A}^1}$ over $\B{A}^1$. Equip the monoid $(\B{A}^1,\cdot)$ act on $X_{\B{A}^1}$ and $C_{\B{A}^1}$ by the product of its actions on $X$ and $C$ and the trivial action on $\B{A}^1$. Then $c_{\B{A}^1}$ is an $(\B{A}^1,\cdot)$-equivariant correspondence, and the induced correspondence $Z_{c_{\B{A}^1}}$ between zero sections is the product of $Z_c$ (see (a)) and $\Id_{\B{A}^1}:\B{A}^1\to \B{A}^1\times \B{A}^1$.

(c) Using (b), for every $t\in\B{A}^1(k)$, we get an $(\B{A}^1,\cdot)$-equivariant correspondence $c_t:C\to X\times X$, which satisfy 
$Z_{c_t}=Z_c$ and $Z_{\Fix(c_t)}=\Fix(Z_c)$ (use (a)).   
\end{Emp}

\begin{Emp} \label{E:cones}
{\bf Cones.} Recall (see \re{famcorr}(a)) that for every $(\B{A}^1,\cdot)$-equivariant scheme $X$, there is a natural projection $\pr_X:X\to Z_X$. We say that $X$ is {\em a cone}, if the projection $\pr_X:X\to Z$ is affine. In concrete terms this means that $X\simeq\C{Spec}(\C{A})$, where $\C{A}=\bigoplus_{n=0}^{\infty} \C{A}_n$ is a graded quasi-coherent $\C{O}_Z$-algebra, where $\C{A}_0=\C{O}_Z$, and each $\C{A}_n$ is a coherent $\C{O}_Z$-module. In this case, the zero section $Z_X\subseteq X$ is automatically closed.

(b) In the situation of (a), the open subscheme $X\sm Z_X\subseteq X$ is $\B{G}_m$-invariant, and the quotient $(X\sm Z_X)/\B{G}_m$ is isomorphic to $Proj(\C{A})$ over $Z_X$, hence is proper over $Z_X$.

(c) Notice that if $c:C\to X\times X$ is an $(\B{A}^1,\cdot)$-equivariant correspondence such that $C$ and $X$ are cones, then 
$\Fix(c)$ is a cone as well (compare \re{equivcor}(a)).
\end{Emp}

Our next goal is to show that in some cases the finiteness assumption in \rp{cones} is automatic.

\begin{Lem} \label{L:verygood}
Let $c:C\to X\times X$ be an $(\B{A}^1,\cdot)$-equivariant correspondence over $k$ such that $X$ is a cone with zero section $Z$, $C$ is a cone with zero section $c_r^{-1}(Z)$, and $\Fix(c|_Z)$ is proper over $k$. Then the set $\{t\in \B{A}^1(k)\,|\,\Fix(c_t)\sm\Fix(c_t|_Z)\neq\emptyset\}$ is finite.
\end{Lem}

\begin{proof}
We let $c_{\B{A}^1}$ be as in \re{equivcor}(b), and set $\Fix(c_{\B{A}^1})':=\Fix(c_{\B{A}^1})\sm \Fix(c_{\B{A}^1}|_{Z_{\B{A}^1}})\subseteq \Fix(c_{\B{A}^1})$. We have to show that the image of the projection
$\pi:\Fix(c_{\B{A}^1})'\to\B{A}^1$ is a finite set. 

Note that the fiber of $\Fix(c_{\B{A}^1})'$ over $0\in\B{A}^1$ is $\Fix(c_0)\sm \Fix(c_0|_Z)=\emptyset$ (by \re{famcorr}(f)).
It thus suffices to show that the image of $\pi$ is closed. By \re{cones}(c), we conclude that $\Fix(c_{\B{A}})$ is a cone, 
while using \re{equivcor}(a),(b) we conclude that $Z_{\Fix(c_{\B{A}})}=\Fix(c_{\B{A}^1}|_{Z_{\B{A}^1}})=\Fix(c|_Z)\times\B{A}^1$.

It now follows from \re{cones}(b) that the open subscheme $\Fix(c_{\B{A}^1})'\subseteq \Fix(c_{\B{A}^1})$ is $\B{G}_m$-invariant, and that $\pi$ factors through the quotient $\Fix(c_{\B{A}^1})'/\B{G}_m$, which is proper over $\Fix(c|_Z)\times\B{A}^1$. Since $\Fix(c|_Z)$ is proper over $k$  by assumption, the projection $\ov{\pi}:\Fix(c_{\B{A}^1})'/\B{G}_m\to\B{A}^1$ is therefore proper. Hence the image of $\ov{\pi}$ is closed, completing the proof.
\end{proof}

\section{Main result}

\begin{Emp} \label{E:normalcone}
{\bf Normal cones} (compare \cite[1.4.1 and Lem 1.4.3]{Va}).

(a) Recall that to a pair $(X,Z)$, where $X$ is a scheme and $Z\subseteq X$ a closed subscheme, one associates the normal cone
$N_Z(X)$ defined to be $N_Z(X)=\C{Spec}(\bigoplus_{n=0}^{\infty}(\C{I}_Z)^n/(\C{I}_Z)^{n+1})$, where $\C{I}_Z\subseteq\C{O}_X$ is the sheaf of ideals of $Z$. By definition, $N_Z(X)$ is a cone in the sense of \re{cones}, and $Z\subseteq N_Z(X)$ is the zero section.

(b) The assignment $(X,Z)\mapsto (N_Z(X),Z)$ is functorial. Namely, every morphism $f:X'\to X$ such that $Z'\subseteq f^{-1}(Z)$ gives rise to an $(\B{A}^1,\cdot)$-equivariant morphism $N_{Z'}(X')\to N_Z(X)$, whose induced morphism between zero sections is $f|_{Z'}:Z'\to Z$.

(c) By (b), every morphism $f:X'\to X$ induces morphism $N_Z(f):N_{f^{-1}(Z)}(X')\to N_Z(X)$, lifting $f|_Z:f^{-1}(Z)\to Z$. 
Moreover, the induced map $N_{f^{-1}(Z)}(X')\to N_Z(X)\times_Z f^{-1}(Z)$ is a closed embedding, and we have an equality  
$N_Z(f)^{-1}(Z)=f^{-1}(Z)\subseteq N_{f^{-1}(Z)}(X')$.
\end{Emp}

The following standard assertion will be important later.

\begin{Lem} \label{L:red}
Assume that $N_Z(X)$ is set-theoretically supported on the zero section, that is,
$N_Z(X)_{\red}=Z_{\red}$. Then $Z_{\red}\subseteq X_{\red}$ is open.
\end{Lem}

\begin{proof}
Since the assertion is local on $X$, we can assume that $X$ is affine. Moreover, replacing $X$ by $X_{\red}$, we can assume that $X$ is reduced. Then our assumption implies that there exists $n$ such that
$I_Z^n=I_Z^{n+1}$. Using Nakayama lemma, we conclude that the localization of $I_Z^n$ at every $x\in Z$ is zero. Thus the localization of $I_Z$ at
every point $z\in Z$ is zero, which implies that $Z\subseteq X$ is open, as claimed.
\end{proof}

\begin{Emp} \label{E:corr}
{\bf Application to correspondences.}
(a) Let $c:C\to X\times X$ be a correspondence, and $Z\subseteq X$ a closed subscheme. Then, by \re{normalcone}, correspondence $c$ gives rise to an  $(\B{A}^1,\cdot)$-equivariant correspondence $N_Z(c): N_{c^{-1}(Z\times Z)}(C)\to N_Z(X)\times N_Z(X)$ such that the induced correspondence between zero sections is $c|_Z:c^{-1}(Z\times Z)\to Z\times Z$.

(b) Combining observations \re{cones}(c) and \re{equivcor}(a), we get that $\Fix(N_Z(c))$  is a cone with zero section $\Fix(c|_Z)$. Moreover, $N_{\Fix(c|_Z)}(\Fix(c))$ is closed subscheme of $\Fix(N_Z(c))$ (see \cite[Cor 1.4.5]{Va}).

(c) By \re{equivcor}(b), for every $t\in\B{A}^1(k)$ we get a correspondence
\[
N_Z(c)_t: N_{c^{-1}(Z\times Z)}(C)\to N_Z(X)\times N_Z(X).
\]  
Moreover, every $\Fix(N_Z(c)_t)$  is a cone with zero section $\Fix(c|_Z)$ (use \re{cones}(c) and \re{equivcor}(c)).
\end{Emp}

\begin{Def} \label{D:good}
Let $c:C\to X\times X$ be a correspondence, and let $Z\subseteq X$ be a closed subscheme.

(a) We say that $c$ {\em has no fixed points in the punctured tubular neighborhood of $Z$}, if correspondence $N_Z(c)$ satisfies $\Fix(N_Z(c))\sm \Fix(c|_Z)=\emptyset$.

(b) We say that $c$ {\em has no almost fixed points in the punctured tubular neighborhood of $Z$}, if $\Fix(N_Z(c))\sm \Fix(c|_Z)=\emptyset$, and the set 
$\{t\in\B{A}^1(k)\,|\,\Fix(N_Z(c)_t)\sm\Fix(c|_Z)\neq\emptyset\}$ is finite.
\end{Def}

\begin{Emp} \label{E:remgood}
{\bf Remarks.} (a) The difference $N_{c^{-1}(Z\times Z)}(C)\sm c^{-1}(Z\times Z)$ can be thought as the punctured tubular neighborhood of
$c^{-1}(Z\times Z)\subseteq C$. Therefore our condition \ref{D:good}(a) means that a point  $y\in N_{c^{-1}(Z\times Z)}(C)\sm c^{-1}(Z\times Z)$ is not a fixed point of $N_Z(c)$, that is, $N_Z(c)_l(y)\neq N_Z(c)_r(y)$.

(b) Condition \ref{D:good}(b) means that there exists an open neighbourhood $U\ni 1$ in $\B{A}^1$ such that for every  $y\in N_{c^{-1}(Z\times Z)}(C)\sm c^{-1}(Z\times Z)$ we have $\mu_t(N_Z(c)_l(y))\neq N_Z(c)_r(y)$ for every $t\in U$.
In other words, $y$ is not an {\em almost fixed} point of $N_Z(c)$.

%(c) It follows from \rl{verygood} that the finiteness condition in \rd{good}(b) is satisfied automatically, if $\Fix(N_Z(c)|_Z)=\Fix(c|_Z)$ is proper over $k$. 
\end{Emp}

\begin{Emp} \label{E:basic}
{\bf The case of a morphism.} (a) Let $f:X\to X$ be a morphism, and let $x\in\Fix(f)$ be a fixed point. We take $c$ be the graph $\Gr_f=(f,\Id_X)$ of $f$, and set $Z:=\{x\}$.

Then $N_x(X):=N_Z(X)$ is a closed conical subset of the tangent space $T_x(X)$, the morphism  $N_x(f):N_x(X)\to N_x(X)$ is $(\B{A}^1,\cdot)$-equivariant, thus $\Fix(N_x(c))=\Fix(N_x(f))$ is a conical subset of $N_x(X)\subseteq T_x(X)$. Thus $\Gr_f$ has no fixed points in the punctured tubular neighborhood of $x$
if and only if set-theoretically $\Fix N_x(f)=\{x\}$.

(b) Let $T_x(f):T_x(X)\to T_x(X)$ be the differential of $f$ at $x$. Then $\Fix T_x(f)=\{x\}$ if and only if the linear map $T_x(f)-\Id:T_x(X)\to T_x(X)$ is invertible, that is, $\Gr_f$ intersects with $\Dt_X$ at $x$ transversally in the strongest possible sense.
In this case, $\Gr_f$ has no fixed points in the punctured tubular neighborhood of $x$ (by (a)).

(c) Assume now that $X$ is smooth at $x$. Then, by (a) and (b), $\Gr_f$ has no fixed points in the punctured tubular neighborhood of $x$ if and only if
$\Gr_f$ intersects with $\Dt_X$ at $x$ transversally.
\end{Emp}

Though the next result is not needed for what follows, it shows that our setting generalizes the one studied in \cite{Va}.

\begin{Lem} \label{L:contr}
Assume that $c$ is contracting near $Z$ in the neighborhood of fixed points in the sense of \cite[2.1.1(c)]{Va}.
Then $c$ has no almost fixed points in the punctured tubular neighborhood of $Z$. Moreover, the subset of $\B{A}^1(k)$, defined in \rd{good}(b), is empty.
\end{Lem}
\begin{proof}
Choose an open neighborhood $W\subseteq C$ of $\Fix(c)$ such that $c|_W$ is contracting near $Z$ (see \cite[2.1.1(b)]{Va}). Then $\Fix(c|_W)=\Fix(c)$, hence we can replace
$c$ by $c|_W$, thus assuming that $c$ is contracting near $Z$. In this case, the set-theoretic image of $N_Z(c)_l:N_{c^{-1}(Z\times Z)}(C)\to
N_Z(X)$ lies in the zero section. Therefore for every $t\in\B{A}^1(k)$ the set-theoretic image of the map $\Fix(N_Z(c)_t)\to N_Z(X)$ lies in the zero section, implying the assertion. 
\end{proof}

By \rl{contr}, the following result is a generalization of \cite[Thm 2.1.3(a)]{Va}.

\begin{Lem} \label{L:good}
Let $c:C\to X\times X$ be a correspondence, which has no fixed points in the punctured tubular neighborhood of $Z\subseteq X$.
Then the closed subscheme $\Fix(c|_Z)_{\red}\subseteq\Fix(c)_{\red}$ is open.
\end{Lem}

\begin{proof}
Recall that we have inclusions
\[
\Fix(c|_Z)_{\red}\subseteq N_{\Fix(c|_Z)}(\Fix(c))_{\red}\subseteq \Fix(N_Z(c))_{\red}
\]
(by \re{normalcone}(b)), while our assumption implies that we have an equality $\Fix(c|_Z)_{\red}= \Fix(N_Z(c))_{\red}$. Therefore we have an equality 
$\Fix(c|_Z)_{\red}=N_{\Fix(c|_Z)}(\Fix(c))_{\red}$, from which our assertion follows by \rl{red}.
\end{proof}

\begin{Emp}
{\bf Notation}
Let $c:C\to X\times X$ be a correspondence,  which has no fixed points in the punctured tubular neighborhood of $Z\subseteq X$.
Then by \rl{good}, $\Fix(c|_Z)\subseteq\Fix(c)$ is an open subset, thus (see \re{locterms}(b)) to every $c$-morphism $u\in\Hom_c(F,F)$ one can associate an element  \[\C{Tr}_{\Fix(c|_Z)}(u)\in H^0(\Fix(c|_Z),K_{\Fix(c|_Z)}).\]
\end{Emp}

Now we are ready to formulate the main result of this note, which by \rl{contr} generalizes \cite[Thm 2.1.3(b)]{Va}.

\begin{Thm} \label{T:main}
Let $c:C\to X\times X$ be a correspondence, and let $Z\subseteq X$ be a $c$-invariant closed subscheme such that
 $c$ has no fixed points in the punctured tubular neighborhood of $Z$.

(a) Assume that $c$ has no almost fixed points in the punctured tubular neighborhood of $Z$. Then for every $c$-morphism
$u\in\Hom_c(\C{F},\C{F})$, we have an equality
\[
\C{Tr}_{\Fix(c|_Z)}(u)=\C{Tr}_{c|_Z}(u|_Z)\in H^0(\Fix(c|_Z),K_{\Fix(c|_Z)}).
\]

(b) Every connected component $\beta$ of $\Fix(c|_Z)$, which is proper over $k$, is also a connected component of $\Fix(c)$. Moreover, for every $c$-morphism $u\in\Hom_c(\C{F},\C{F})$, we have an equality 
\[
LT_{\beta}(u)=LT_{\beta}(u|_Z).
\]
\end{Thm}

As an application, we now deduce the result, stated in the introduction.

\begin{Cor} \label{C:main}
Let $f:X\to X$ be a morphism, and let $x\in\Fix(f)$ be a fixed point such that the induced map of normal cones $N_x(f):N_x(X)\to N_x(X)$ has no non-zero fixed points. Then:

(a) Point $x$ is an isolated fixed point of $f$.

(b) For every morphism $u:f^*\C{F}\to\C{F}$ with $\C{F}\in D_{ctf}^b(X,\La)$, we have an equality
$LT_{x}(u)=\Tr(u_x)$. In particular, if $\C{F}=\La$ and $u$ is the identity, then $LT_{x}(u)=1$.
\end{Cor}

\begin{proof}
As it was observed in \re{basic}(a), the assumption implies that $\{x\}\subseteq X$ is a closed $\Gr_f$-invariant subscheme, and 
correspondence $\Gr_f$ has no fixed points in the punctured tubular neighborhood of $\{x\}$. Therefore part (a) follows from \rl{good}, while the first assertion of (b) is an immediate corollary of \rt{main}. The second assertion of (b) now follows from the obvious observation that $\Tr(u_x)=1$.
\end{proof}

\section{The case of group actions}

\begin{Lem} \label{L:diagonal}
Let $D$ be a diagonalizable algebraic group acting on a scheme $X$, and let $Z\subseteq X$ be a $D$-invariant closed subscheme. Then $D$ acts on the normal cone $N_Z(X)$, and the induced morphism $N_{Z^D}(X^D)\to N_Z(X)^D$ on $D$-fixed points is an isomorphism.
\end{Lem}

\begin{proof}
By the functoriality of the normal cone (see \re{normalcone}(b)), $D$ acts on the normal cone $N_Z(X)$, so it remains to show that the map $N_{Z^D}(X^D)\to N_Z(X)^D$ is an isomorphism.

Assume first that $D$ is a finite group of order prime to the characteristic of $k$. Note that
every $z\in Z^D$ has a $D$-invariant open affine neighbourhood $U\subseteq X$. Thus, replacing $X$ by $U$ and $Z$ by $Z\cap U$, we can assume that $X$ and $Z$ are affine. Then we have to show that the map $k[N_Z(X)]_D\cong k[N_Z(X)^D]\to k[N_{Z^D}(X^D)]$ is an isomorphism.

The latter assertion follows from the fact that the functor of coinvariants $M\mapsto M_D$ is exact on $k[D]$-modules. Namely, the exactness of $(-)_D$ implies that the isomorphism $k[X]_D\isom k[X^D]$ induces
an isomorphism $(I_Z)_D\isom I_{Z^D}$, and the rest is easy.

To show the case of a general $D$, notice that the set of torsion elements $D_{\operatorname{tor}}\subseteq D$ is Zariski dense. Since $X,Z$ and $N_Z(X)$ are Noetherian, therefore there exists a finite subgroup $D'\subseteq D$ such that $X^D=X^{D'}$ and similarly for $Z$ and $N_Z(X)$. Therefore
the assertion for $D$ follows from that for $D'$, shown above.
\end{proof}

\begin{Cor} \label{C:diagonal}
Let $D$ be as in \rl{diagonal}, let $g\in D$, and let $Z\subseteq X$ be a $g$-invariant closed subscheme. Then $g$ induces an endomorphism of the normal cone $N_Z(X)$, and the induced morphism $N_{Z^g}(X^g)\to N_Z(X)^g$ on $g$-fixed points is an isomorphism.
\end{Cor}

\begin{proof}
Let $D':=\ov{\lan g\ran}\subseteq D$ be the Zariski closure of the cyclic group $\lan g\ran\subseteq D$. Then $D'$ is a diagonalizable group, and
we have an equality $X^g=X^{D'}$ and similarly for $Z^g$ and $N_Z(X)^g$. Thus the assertion follows from \rl{diagonal} for $D'$.
\end{proof}
%(Actually, we deduced the proof of the general case from this particular case.)

\begin{Emp} \label{E:appl}
{\bf Example.} Let $g:X\to X$ be an automorphism of finite order, which is prime to the characteristic of $k$.
Then the cyclic group $\lan g\ran\subseteq\Aut(X)$ is a diagonalizable group, thus \rco{diagonal} applies in this case. So for every $g$-invariant closed subscheme $Z\subseteq X$, the natural morphism $N_{Z^g}(X^g)\to N_Z(X)^g$ is an isomorphism.
\end{Emp}

As a consequence we get a class of examples, when the condition of \rd{good}(a) is satisfied.

\begin{Cor} \label{C:dl}
Let $G$ be a linear algebraic group acting on $X$.

(a) Let $g\in G$, let $\ov{\lan g\ran}$ be the Zariski closure of the cyclic group generated by $g$, let $s\in\ov{\lan g\ran}$ be a semisimple element, and let $Z\subseteq X$ be an $s$-invariant closed subscheme such that
$(X\sm Z)^s=\emptyset$. Then $g$ has no fixed points in the punctured tubular neighborhood of $Z$.

(b) Let $g\in G$ be semisimple, and let $Z\subseteq X$ be a $g$-invariant closed subscheme such that
$(X\sm Z)^g=\emptyset$. Then $g$ has no fixed points in the punctured tubular neighborhood of $Z$.
\end{Cor}

\begin{proof}
(a) We have to show that $N_Z(X)^g\sm Z=\emptyset$. By assumption, we have $N_Z(X)^g\subseteq N_Z(X)^s$. Therefore, it suffices to show that $N_Z(X)^s\sm Z=N_Z(X)^s\sm Z^s=\emptyset$. Since $s$ is semisimple, we conclude from \rco{diagonal} that $N_Z(X)^s=N_{Z^s}(X^s)$. Since $(X^s)_{\red}=(Z^s)_{\red}$, by assumption, we conclude that  $N_{Z^s}(X^s)_{\red}=(Z^s)_{\red}$, implying the assertion.

(b) is a particular case of (a).
\end{proof}

\begin{Emp} \label{E:jordan}
{\bf Example.}
An important particular case of \rco{dl}(a) is when $s=g_s$ is the semisimple part of $g$, that is, $g=g_sg_u$ is the Jordan decomposition.
%Notice that if $g$ is of finite order, and the field $k$ is of positive characteristic $p$, then such an $s$ can be characterized as a unique power of $g$ such that the order of $gs^{-1}$ is a power of $p$.
%Let $g:X\to X$ be an automorphism of $X$ of finite order. Then $X$ is equipped with an action of the finite cyclic group $G=\lan g\ran$.
%Moreover, if $g=su$ is the Jordan decomposition of $g$, then $s,u:X\to X$ are the unique automorphisms such that
%$su=us=g$, $s$ is of order prime to $p$, and $u$ is of order power of $p$.
%Then we have an equality $\Tr(g,H^*_c(X))=\Tr(u,H^*_c(X^s))$.
\end{Emp}

The following result gives a version of \rco{main}, whose assumptions are easier to check.

\begin{Cor} \label{C:eqlocterms}
Let $G$ and $g\in G$ be as in \rco{dl}(b), and let $x\in X^g$ be an isolated fixed point of $g$.
Then the induced map of normal cones $g:N_x(X)\to N_x(X)$ has no non-zero fixed points. Therefore
for every morphism $u:g^*\C{F}\to\C{F}$ with $\C{F}\in D_{ctf}^b(X,\La)$, we have an equality
\[
LT_{x}(u)=\Tr(u_x).
\]
\end{Cor}

\begin{proof}
The first assertion follows from \rco{dl}(b), while the second one follows from \rco{main}(b).
\end{proof}

\begin{Emp} \label{E:KW}
{\bf An application.} \rco{eqlocterms} is used in the work of D. Hansen, T. Kaletha and J. Weinstein (see \cite[Prop. 5.6.2]{HKW}).
\end{Emp}

As a further application, we get a slight generalization of the Deligne--Lusztig trace formula.

\begin{Emp} \label{E:not}
{\bf Notation.} To every proper endomorphism $f:X\to X$ and a morphism $u:f^*\C{F}\to\C{F}$ with $\C{F}\in D_{ctf}^b(X,\La)$, one associates an endomorphism $R\Gm_c(u): R\Gm_c(X,\C{F})\to R\Gm_c(X,\C{F})$ (compare \cite[1.1.7]{Va}).

Moreover, for an $f$-invariant closed subscheme $Z\subseteq X$, we set $U:=X\sm Z$ and form endomorphisms 
%$\C{F}|_Z\in D_{ctf}^b(Z,\La)$, $u|_Z:f^*(\C{F}|_Z)\to\C{F}|_Z$ and
$R\Gm_c(u|_Z): R\Gm_c(Z,\C{F}|_Z)\to R\Gm_c(Z,\C{F}|_Z)$ and $R\Gm_c(u|_U): R\Gm_c(U,\C{F}|_U)\to R\Gm_c(U,\C{F}|_U)$ (compare \re{cohcorr}(d)).
\end{Emp}

\begin{Thm} \label{T:dl}
Let $G$ be a linear algebraic group acting on a separated scheme $X$, let $g\in G$ be such that  $X$ has a $g$-equivariant compactification, and
let $s\in\ov{\lan g\ran}$ be a semisimple element.

Then $X^s\subseteq X$ is a closed $g$-invariant subscheme, and for every
morphism  $u:g^*\C{F}\to\C{F}$ with $\C{F}\in D_{ctf}^b(X,\La)$, we have an equality of traces
$\Tr(R\Gm_c(u))=\Tr(R\Gm_c(u|_{X^s}))$ (see \re{not}).
\end{Thm}

\begin{proof}
Using the equality
$\Tr(R\Gm_c(u))=\Tr(R\Gm_c(u|_{X^s}))+\Tr(R\Gm_c(u|_{X\sm X^s}))$, it remains to show that $\Tr(R\Gm_c(u|_{X\sm X^s}))=0$. Thus, replacing $X$ by $X\sm X^s$ and $u$ by $u|_{X\sm X^s}$, we may assume that $X^s=\emptyset$, and
we have to show that $\Tr(R\Gm_c(u))=0$.

Choose a $g$-equivariant compactification $\ov{X}$ of $X$, and set $Z:=(\ov{X}\sm X)_{\red}$. Let $j:X\hra\ov{X}$ be the open inclusion, and set $\ov{\C{F}}:=j_!\C{F}\in D_c^b(\ov{X},\qlbar)$. Since $X\subseteq\ov{X}$ is $g$-invariant, our morphism $u$ extends to a morphism $\ov{u}=j_!(u):g^*\ov{\C{F}}\to\ov{\C{F}}$, and we have an equality $\Tr(R\Gm_c(u))=\Tr(R\Gm_c(\ov{u}))$ (compare \cite[1.1.7]{Va}).
Thus, since $\ov{X}$ is proper, the Lefschetz--Verdier trace formula says that
\begin{equation*} \label{Eq:LTF}
\Tr(R\Gm_c(u))=\Tr(R\Gm_c(\ov{u}))=\sum_{\beta\in\pi_0(\ov{X}^g)}LT_{\beta}(\ov{u}),
\end{equation*}
so it suffices to show that each local term $LT_{\beta}(\ov{u})$ vanishes.

Since $X^g\subseteq X^s=\emptyset$, we have $(\ov{X}^g)_{\red}=(Z^g)_{\red}$. Thus
every $\beta$ is a connected component of $Z^g$. In addition, $g$ have no fixed points in the punctured neighborhood of $Z$ (by \rco{dl}(a)).
Therefore, by \rt{main}, we have an equality $LT_{\beta}(\ov{u})=LT_{\beta}(\ov{u}|_Z)$. However, the latter expression vanishes, because $\ov{\C{F}}|_Z=0$, therefore $\ov{u}|_Z=0$. This completes the proof.
\end{proof}

\begin{Cor} \label{C:dl2}
Let $X$ be a scheme over $k$, let $g:X\to X$ be an automorphism of finite order, and let $s$ be a power of $g$ such that $s$ is of order
prime to the characteristic of $k$. Then for every
morphism  $u:g^*\C{F}\to\C{F}$ with $\C{F}\in D_{ctf}^b(X,\La)$, we have an equality of traces
\[
\Tr(R\Gm_c(u))=\Tr(R\Gm_c(u|_{X^s})).
\]
\end{Cor}

\begin{proof}
Notice that since $g$ is an automorphism of finite order, $X$ has a $g$-invariant open dense subscheme $U$.
Using additivity of traces $\Tr(R\Gm_c(u))=\Tr(R\Gm_c(u|_{U}))+\Tr(R\Gm_c(u|_{X\sm U}))$, and Noetherian induction on
$X$, we can therefore assume that $X$ is affine. Then $X$ has a $g$-equivariant compactification, so the assertion follows from \rt{dl}. 
\end{proof}

\begin{Emp}
{\bf Example.} Applying \rco{dl2} in the case when $\C{F}=\qlbar$ and $u$ is the identity, we recover the identity
$\Tr(g,R\Gm_c(X,\ql))=\Tr(g,R\Gm_c(X^s,\ql))$, proven in \cite[Thm. 3.2]{DL}.
\end{Emp}

\section{Proof of \rt{main}}

\begin{Emp} \label{E:ncone}
{\bf Deformation to the normal cone} (see \cite[1.4.1 and Lem 1.4.3]{Va}).

Let $R=k[t]_{(t)}$ be the localization of $k[t]$ at $(t)$, set $\C{D}:=\Spec R$,
and let $\eta$ and $s$ be the generic and the special points of $\C{D}$, respectively.

(a) Let $X$ be a scheme over $k$, and $Z\subseteq X$ a closed subscheme. Recall (\cite[1.4.1]{Va}) that to this data one associate a scheme $\wt{X}_{Z}$ over $X_{\C{D}}:=X\times \C{D}$, whose generic fiber (that is, fiber over $\eta\in\C{D}$) is $X_{\eta}:=X\times\eta$, and special fiber is the normal cone $N_Z(X)$.

(b) We have a canonical closed embedding $Z_{\C{D}}\hra \wt{X}_Z$, whose generic fiber is the embedding $Z_{\eta}\hra X_{\eta}$, and special
fiber is $Z\hra N_Z(X)$.

(c)  The assignment $(X,Z)\mapsto \wt{X}_Z$ is functorial, that is, for every morphism $f:(X',Z')\to (X,Z)$ there exists a unique morphism
$\wt{X'}_{Z'}\to\wt{X}_{Z}$ lifting $f_{\C{D}}$ (see \cite[Lem 1.4.3]{Va}). In particular, $f$ gives rise to a canonical morphism $N_{Z'}(X')\to N_Z(X)$ from \re{normalcone}(b).

(d) Let $c:C\to X\times X$ be a correspondence, and let $Z\subseteq X$ be a closed subscheme. Then, by (c), one gets the correspondence
$\wt{c}_Z: \wt{C}_{c^{-1}(Z\times Z)}(C)\to \wt{X}_Z\times \wt{X}_Z$ over $\C{D}$, whose generic fiber is $c_{\eta}$, and special fiber is the correspondence $N_Z(c): N_{c^{-1}(Z\times Z)}(C)\to N_Z(X)\times N_Z(X)$ from \re{corr}(a). 

(e) By (b), we have a canonical closed embedding $\Fix(c|_Z)_{\C{D}}\hra \Fix(\wt{c}_Z)$ over $\C{D}$, whose generic fiber is the embedding $\Fix(c|_Z)_{\eta}\hra \Fix(c)_{\eta}$, and special fiber is $\Fix(c|_Z)\hra \Fix(N_Z(c))$.
\end{Emp}

\begin{Emp} \label{E:specnc}
{\bf Specialization to the normal cone.} Assume that we are in the situation of \re{ncone}.

(a) As in \cite[1.3.2]{Va}, we have a canonical functor $sp_{\wt{X}_Z}:D^b_{ctf}(X,\La)\to D^b_{ctf}(N_Z(X),\La)$. Moreover, for every $\C{F}\in D^b_{ctf}(X,\La)$,
we have a canonical morphism
\[
sp_{\wt{c}_Z}:\Hom_c (\C{F},\C{F})\to  \Hom_{N_Z(c)}(sp_{\wt{X}_Z}(\C{F}), sp_{\wt{X}_Z}(\C{F})).
\]

(b) As in \cite[1.3.3(b)]{Va}, we have a canonical specialization map
\[
sp_{\Fix(\wt{c}_Z)}:H^0(\Fix(c), K_{\Fix(c)})\to H^0(\Fix(N_Z(c)), K_{\Fix(N_Z(s))}),
\]
which is an isomorphism when $\Fix(\wt{c}_Z)\to\C{D}$ is a topologically constant family.

(c) Applying \cite[Prop 1.3.5]{Va}  in this case, we conclude that for every $\C{F}\in D_{ctf}^b(X,\La)$, the following diagram
is commutative
\begin{equation} \label{Eq:spec}
\CD
      \Hom_c (\C{F},\C{F}) @>\C{Tr}_c>> H^0(\Fix(c), K_{\Fix(c)})\\
      @V{sp_{\wt{c}_Z}}VV            @V{sp_{\Fix(\wt{c}_Z)}}VV\\
      \Hom_{N_Z(c)}(sp_{\wt{X}_Z}(\C{F}), sp_{\wt{X}_Z}(\C{F}))
      @>\C{Tr}_{N_Z(c)}>>  H^0(\Fix(N_Z(c)), K_{\Fix(N_Z(s))}).
\endCD
\end{equation}
%

%(e) In particular, we have a closed embedding $Z\times\B{A}^1=\wt{Z}_Z\hra \wt{X}_Z$.
\end{Emp}

Now we are ready to prove \rt{main}, mostly repeating the argument of \cite[Thm 2.1.3(b)]{Va}.

\begin{Emp} \label{E:pfmaina}
{\bf Proof of \rt{main}(a).}

{\bf Step 1.} We may assume that $\Fix(c)_{\red}=\Fix(c|_Z)_{\red}$.

\begin{proof}
By \rl{good}, there exists an open subscheme $W\subseteq C$ such that $W\cap\Fix(c)_{\red}=\Fix(c|_Z)_{\red}$. Replacing
$c$ by $c|_W$ and $u$ by $u|_W$, we can assume that $\Fix(c)_{\red}=\Fix(c|_Z)_{\red}$.
\end{proof}

{\bf Step 2.} We may assume that $\C{F}|_Z\simeq 0$, and it suffices to show that in this case $\C{Tr}_c(u)=0$.

\begin{proof}
Set $U:=X\sm Z$, and let $i:Z\hra X$ and $j:U\hra X$ be the embeddings. Since $Z$ is $c$-invariant, one can associate to $u$
two $c$-morphisms $[i_Z]_!(u|_Z)\in\Hom_c(i_!(\C{F}|_Z),i_!(\C{F}|_Z))$ and $[j_U]_!(u|_U)\in\Hom_c(j_!(\C{F}|_U),j_!(\C{F}|_U))$ (see \cite[1.5.9]{Va}). Then, by the additivity of the trace map \cite[Prop. 1.5.10]{Va}, we conclude that
 \[
 \C{Tr}_c(u)=\C{Tr}_c([i_Z]_!(u|_Z))+\C{Tr}_c([j_U]_!(u|_U)).
 \]

Moreover, using the assumption $\Fix(c|_Z)_{\red}=\Fix(c)_{\red}$ and the commutativity of the trace map with closed embeddings \cite[Prop 1.2.5]{Va}, we conclude that
\[
\C{Tr}_c([i_Z]_!(u|_Z))=\C{Tr}_{c|_Z}(u|_Z).
\] Thus it remains to show that $\C{Tr}_c([j_U]_!(u|_U))=0$.
For this we can replace $\C{F}$ by $j_!(\C{F}|_U)$ and $u$ by $[j_U]_!(u|_U)$. In this case, $\C{F}|_Z\simeq 0$, and it remains to show that
$\C{Tr}_c(u)=0$ as claimed.
\end{proof}

{\bf Step 3. Specialization to the normal cone.} By the commutative diagram \form{spec}, we have an equality
\[
\C{Tr}_{N_Z(c)}(sp_{\wt{c}_Z}(u))=sp_{Fix(\wt{c}_Z)}(\C{Tr}_c(u)).
\]
Thus to show the vanishing of $\C{Tr}_c(u)$, it suffices to show that

(i) the map $sp_{\Fix(\wt{c}_Z)}$ is an isomorphism, and

(ii) we have $\C{Tr}_{N_Z(c)}(sp_{\wt{c}_Z}(u))=0$.

\vskip 8truept

{\bf Step 4. Proof of Step 3(i).} By observation \re{specnc}(b), it suffices to show that the closed embedding $\Fix(c|_Z)_{\C{D},\red}\hra \Fix(\wt{c}_Z)_{\red}$ (see \re{ncone}(b)) is an isomorphism.
Moreover, we can check separately the corresponding assertions for generic and special fibers.

For generic fibers, the assertions follows from our assumption
$\Fix(c)_{\red}=\Fix(c|_Z)_{\red}$ (see Step 1), while the assertion for special fibers  $\Fix(c|_Z)_{\red}=\Fix(N_Z(c))_{\red}$
follows from our assumption that $c$ has no
fixed points in the punctured tubular neighborhood of $Z$.

\vskip 8truept

{\bf Step 5. Proof of Step 3(ii).} By a standard reduction, one can assume that $\La$ is finite. We are going to deduce the assertion from  \rp{cones} applied to the correspondence $N_Z(c)$ and a weakly $\B{G}_m$-equivariant $sp_{\wt{X}_Z}(\C{F})\in D_{ctf}(N_Z(X),\La)$.

Note that the zero section $Z\subseteq N_Z(X)$ is closed (by \re{normalcone}(a)). Next, since $Z$ is $c$-invariant, we have $c^{-1}(Z\times Z)=c_r^{-1}(Z)$. Therefore if follows by \re{normalcone}(c) that $Z\subseteq N_Z(X)$ is $N_Z(c)$-invariant, and
the correspondence $N_Z(c)_t|_Z$ is identified with $Z_{N_Z(c)}=c|_Z$. Since $c$ has no almost fixed points in the punctured tubular neighborhood of $Z$, we therefore conclude that $N_Z(c)$ satisfies the assumptions of  \rp{cones}. So it remains to show that $sp_{\wt{X}_Z}(\C{F})|_Z\simeq 0$ and that $sp_{\wt{X}_Z}(\C{F})$ is weakly $\B{G}_m$-equivariant with respect to the $n$-twisted action for some $n$.

Both assertions follow from results of Verdier \cite{Ve}. Namely, the vanishing assertion follows from isomorphism $sp_{\wt{X}_Z}(\C{F})|_Z\simeq\C{F}|_Z$  (see \cite[$\S$8, (SP5)]{Ve} or \cite[Prop. 1.4.2]{Va}) and our assumption $\C{F}|_Z\simeq 0$ (see Step 2). The equivariance assertion follows from the fact that $sp_{\wt{X}_Z}(\C{F})$ is monodromic (see \cite[$\S$8, (SP1)]{Ve}), because  $\La$ is finite
(use \cite[Prop 5.1]{Ve}).
%In other words, all the assumptions of \rp{cones} are satisfied, thus $\C{Tr}_{N_Z(c)}(sp_{\wt{c}_Z}(u))=0$, as claimed.
\end{Emp}

\begin{Emp}
{\bf Proof of \rt{main}(b).}
The first assertion follows from \rl{good}. To show the second one,  choose an open subscheme $W\subseteq C$ such that $W\cap\Fix(c)_{\red}=\beta_{\red}$. Replacing $c$ by $c|_W$, we can assume that $\beta_{\red}=\Fix(c)_{\red}=\Fix(c|_Z)_{\red}$, 
thus $\Fix(c|_Z)$ is proper over $k$. 

As it was already observed in Step 5 of \re{pfmaina}, the correspondence $N_Z(c)|_Z$ is identified with $c|_Z$.
Thus $\Fix(N_Z(c)|_Z)=\Fix(c|_Z)$ is proper over $k$. It now follows from \rl{verygood} that the finiteness condition in \rd{good}(b) is satisfied automatically, therefore $c$ has no almost fixed points in the tubular neighborhood of $Z$ (see \re{remgood}(c)). Now the equality $LT_{\beta}(u)=LT_{\beta}(u|_Z)$ follows from obvious equalities $\C{Tr}_{\beta}(u)=\C{Tr}_{\Fix(c|_Z)}(u),
\C{Tr}_{\beta}(u|_Z)=\Tr_{c|_Z}(u|_Z)$ and part (a).
\end{Emp}

\section{Proof of \rp{const}} \label{S:rpconst}

We are going to reduce the result from the assertion that trace maps commute with nearby cycles.

\begin{Emp} \label{E:setup}
{\bf Set up.} Let $\C{D}$ be a spectrum of a discrete valuation ring over $k$ with residue field $k$, and $f:X\to\C{D}$ a morphism of schemes of finite type.

(a) Let $\eta$, $\ov{\eta}$ and $s$ be the generic, the geometrically generic and the special point of $\C{D}$, respectively.
We denote by $X_{\eta},X_{\ov{\eta}}$ and $X_s$ the generic, the geometric generic and the special fiber of  $X$, respectively,
and let $i_{{\eta}}:X_{{\eta}}\to X$, $i_{\ov{\eta}}:X_{\ov{\eta}}\to X$, $i_s:X_s\to X$  and $\pi_{\eta}:X_{\ov{\eta}}\to X_{\eta}$ be the canonical morphisms.

(b) For every $\C{F}\in D(X,\La)$, we set $\C{F}_{\eta}:=i^*_{\eta}(\C{F}), \C{F}_{\ov{\eta}}:=i^*_{\ov{\eta}}(\C{F})$ and  $\C{F}_{s}:=i^*_{s}(\C{F})$. For every  $\C{F}_{\eta}\in  D(X_{\eta},\La)$, we set  $\C{F}_{\ov{\eta}}:=\pi_{\eta}^*(\C{F}_{\eta})$.

(c) Let $\Psi=\Psi_X:D_{ctf}^b(X_{\eta},\La)\to D_{ctf}^b(X_{s},\La)$ be the nearby cycle functor. By definition, it is defined by the formula
$\Psi_X(\C{F}_{\eta}):=i^*_s i_{\ov{\eta}*}(\C{F}_{\ov{\eta}})$.

(d) Consider functor $\ov{\Psi}_X:=i^*_s\circ i_{\ov{\eta}*}:D(X_{\ov{\eta}},\La)\to D(X_{s},\La)$.  Then we have
 $\Psi_X(\C{F}_{\eta})=\ov{\Psi}_X(\C{F}_{\ov{\eta}})$ for all $\C{F}_{\eta}\in D_{ctf}^b(X_{\eta},\La)$.
\end{Emp}

\begin{Emp} \label{E:ULA}
{\bf ULA sheaves.} Assume that we are in the situation of \re{setup}.

(a) We have a canonical isomorphism $\Psi_{{X}}\circ i^*_{\eta}\simeq i_s^*\circ i_{\ov{\eta}*}\circ i^*_{\ov{\eta}}$ of functors $D_{ctf}^b(X,\La)\to D_{ctf}^b(X_s,\La)$.
In particular, the unit map $\Id\to i_{\ov{\eta}*}\circ i^*_{\ov{\eta}}$ induces a morphism  of functors 
$i_s^*\to \Psi_{{X}}\circ i^*_{\eta}=\ov{\Psi}_{{X}}\circ i^*_{\ov{\eta}}$.

(b) Note that if $\C{F}\in  D_{ctf}^b(X,\La)$ is ULA over $\C{D}$, then the induced morphism
\[
\C{F}_s=i_s^*(\C{F})\to \Psi_{{X}}i^*_{\eta}(\C{F})=\Psi_X(\C{F}_{\eta})=\ov{\Psi}_X(\C{F}_{\ov{\eta}})
\]
 is an isomorphism. In particular, we
have a canonical isomorphism $\La_{s}\simeq \ov{\Psi}_{\C{D}}(\La_{\ov{\eta}})$.
\end{Emp}

\begin{Emp} \label{E:nearbyspec}
{\bf Construction.} Assume that we are in the situation of \re{setup}.

(a) For every $\C{F}_{\ov{\eta}}\in D(X_{\ov{\eta}},\La)$, consider composition
\[
R\Gm(X_{\ov{\eta}}, \C{F}_{\ov{\eta}})\simeq R\Gm(X, i_{\ov{\eta}*}(\C{F}_{\ov{\eta}}))\overset{i_s^*}{\lra}  R\Gm(X_s, i_s^*i_{\ov{\eta}*}(\C{F}_{\ov{\eta}}))=R\Gm(X_s, \ov{\Psi}_X(\C{F}_{\ov{\eta}})).
\]

(b) Consider canonical morphism $\ov{\Psi}_X(K_{X_{\ov{\eta}}})\to K_{X_s}$, defined as a composition
\[
\ov{\Psi}_X(K_{X_{\ov{\eta}}})=\ov{\Psi}_X(f_{\ov{\eta}}^!(\La_{\ov{\eta}}))\overset{BC}{\lra}f_{s}^!(\ov{\Psi}_{\C{D}}(\La_{\ov{\eta}}))\simeq f_{s}^!(\La_s)=K_{X_s}.
\]

(c) Denote by $\ov{\Sp}_X$ the composition
\[
R\Gm(X_{\ov{\eta}}, K_{X_{\ov{\eta}}})\overset{(a)}{\lra} R\Gm(X_s, \ov{\Psi}_X(K_{X_{\ov{\eta}}}))\overset{(b)}{\lra} R\Gm(X_s,K_{X_s}).
\]

(d) Using the observation $K_{X_{\ov{\eta}}}\simeq \pi_{\eta}^*(K_{X_{{\eta}}})$, we denote by $\Sp_X$ the composition
\[
R\Gm(X_{{\eta}}, K_{X_{\ov{\eta}}})\overset{\pi_{\eta}^*}{\lra} R\Gm(X_{\ov{\eta}}, K_{X_{\ov{\eta}}})\overset{\ov{\Sp}_X}{\lra} R\Gm(X_s, K_{X_s}).
\]
\end{Emp}

\begin{Lem} \label{L:toptriv}
Assume that $f:X\to\C{D}$ is a topologically constant family (see \re{family}). Then the specialization map
 $\ov{\Sp}_X:R\Gm(X_{\ov{\eta}}, K_{X_{\ov{\eta}}})\to R\Gm(X_s, K_{X_s})$ of \re{nearbyspec}(c) coincides with the canonical identification of \rcl{family}.
\end{Lem}

\begin{proof}
Though the assertion follows by straightforward unwinding the definitions, we sketch the argument for the convenience of the reader.

As in the proof of \rcl{family}, we set $K_{X/\C{D}}:=f^!(\La_{\C{D}})$ and $\C{F}:=f_*(K_{X/\C{D}})$. Consider diagram
\begin{equation} \label{Eq:trivfam}
\begin{CD}
\C{F}_{\ov{\eta}}=R\Gm(\ov{\eta},\C{F}_{\ov{\eta}}) @>\re{nearbyspec}(a)>> R\Gm(s, \ov{\Psi}_{\C{D}}(\C{F}_{\ov{\eta}})) @<\re{ULA}(b)<< R\Gm(s,\C{F}_s)=\C{F}_s\\
@VBC_*VV @VBC_*VV @VBC_*VV\\
R\Gm(X_{\ov{\eta}}, (K_{X/\C{D}})_{\ov{\eta}}) @>\re{nearbyspec}(a)>> R\Gm(X_s, \ov{\Psi}_X((K_{X/\C{D}})_{\ov{\eta}})) @<\re{ULA}(b)<< R\Gm(X_s,(K_{X/\C{D}})_s)\\
@VBC^*VV @VBC^*VV @VBC^*VV\\
R\Gm(X_{\ov{\eta}}, K_{X_{\ov{\eta}}}) @>\re{nearbyspec}(a)>> R\Gm(X_s, \ov{\Psi}_X(K_{X_{\ov{\eta}}})) @>\re{nearbyspec}(b)>> R\Gm(X_s,K_{X_s}),
\end{CD}
\end{equation}
where

$\bullet$ maps denoted by $BC_*$ are induced by the (base change) isomorphisms $\C{F}_{\ov{\eta}}\isom f_{\ov{\eta}*}((K_{X/\C{D}})_{\ov{\eta}})$,
$\C{F}_{s}\isom f_{s*}((K_{X/\C{D}})_s)$ and base change morphisms, while

$\bullet$ maps denoted by $BC^*$ are induced by the (base change) isomorphisms $(K_{X/\C{D}})_{\ov{\eta}}\isom K_{X_{\ov{\eta}}}$ and $(K_{X/\C{D}})_{s}\isom K_{X_s}$.

We claim that the diagram \form{trivfam} is commutative. Since top left, top right and bottom left inner squares are commutative by functoriality, it remain to show the commutativity of the right bottom inner square. In other words, it suffices to show the commutativity of the following diagram
\[
\begin{CD}
\ov{\Psi}_X((K_{X/\C{D}})_{\ov{\eta}}) @<\re{ULA}(b)<< (K_{X/\C{D}})_s\\
@VBC^*VV @VBC^*VV\\
\ov{\Psi}_X(K_{X_{\ov{\eta}}}) @>\re{nearbyspec}(b)>> K_{X_s}.
\end{CD}
\]
Using identity $K_{X/\C{D}}=f^!(\La_{\C{D}})$, it suffices to show the commutativity of the diagram
\[
\begin{CD}
i_s^*i_{\ov{\eta}*}i^*_{\ov{\eta}}f^! @<unit << i_s^*f^! @>BC>>f_s^!i_s^*\\
@VBCVV @VunitVV    @VVunitV\\
i_s^*i_{\ov{\eta}*}f^! i^*_{\ov{\eta}} @>BC>> i_s^*f^!i_{\ov{\eta}*}i^*_{\ov{\eta}} @>BC>> f_s^!i_s^*i_{\ov{\eta}*}i^*_{\ov{\eta}},
\end{CD}
\]
which is standard.

By the commutativity of \form{trivfam}, it remains to show that the top arrow
\[
\C{F}_{\ov{\eta}}=R\Gm(\ov{\eta},\C{F}_{\ov{\eta}})\to R\Gm(s, \ov{\Psi}_{\C{D}}(\C{F}_{\ov{\eta}}))\simeq R\Gm(s,\C{F}_s)=\C{F}_s
\]
of \form{trivfam} equals the inverse of the specialization map
\[
\C{F}_s=R\Gm(s,\C{F}_s)\simeq R\Gm(\C{D},\C{F})\overset{i_{\ov{\eta}}^*}{\lra}R\Gm(\ov{\eta},\C{F}_{\ov{\eta}})=\C{F}_{\ov{\eta}}.
\]

But this follows from the commutativity of the following diagram
\[
\begin{CD}
R\Gm(\ov{\eta},\C{F}_{\ov{\eta}})@<i_{\ov{\eta}}^*<< R\Gm(\C{D},\C{F}) @>i_s^*>> R\Gm(s,\C{F}_s) \\
@| @Vunit VV @VunitVV \\
R\Gm(\ov{\eta},\C{F}_{\ov{\eta}})@= R\Gm(\C{D}, i_{\ov{\eta}*}(\C{F}_{\ov{\eta}})) @>i_s^*>> R\Gm(s,i_s^* i_{\ov{\eta}*}(\C{F}_{\ov{\eta}})).
\end{CD}
\]

\end{proof}

\begin{Emp} \label{E:nearby}
{\bf Specialization of cohomological correspondences.} Let ${c}:{C}\to {X}\times {X}$ be a
correspondence over $\C{D}$, let ${c}_{\eta}:{C}_{\eta}\to {X}_{\eta}\times {X}_{\eta}$, ${c}_{\ov{\eta}}:{C}_{\ov{\eta}}\to {X}_{\ov{\eta}}\times {X}_{\ov{\eta}}$ and
${c}_{s}:{C}_{s}\to {X}_{s}\times {X}_{s}$ be the generic, the geometric generic and the special fibers of ${c}$, respectively. Fix $\C{F}_{\eta}\in D_{ctf}^b({X}_{\eta},\La)$.

(a) Using the fact that the projection $\pi_{\eta}:\ov{\eta}\to\eta$ is pro-\'etale, we have the following commutative diagram
\begin{equation*} \label{Eq:extscal}
\CD
      \Hom_{{c}_{\eta}}(\C{F}_{\eta},\C{F}_{\eta})
  @>\C{Tr}_{{c}_{\eta}}>> H^0(\Fix({c}_{\eta}), K_{\Fix({c}_{\eta})})\\
      @V\pi_{\eta}^*VV            @V\pi_{\eta}^*VV\\
       \Hom_{{c}_{\ov{\eta}}}(\C{F}_{\ov{\eta}},\C{F}_{\ov{\eta}})
  @>\C{Tr}_{{c}_{\ov{\eta}}}>> H^0(\Fix({c}_{\ov{\eta}}), K_{\Fix({c}_{\ov{\eta}})}).
\endCD
\end{equation*}

(b) Consider the map
\[
\Psi_{{c}}:\Hom_{{c}_{\eta}}(\C{F}_{\eta},\C{F}_{\eta})\to \Hom_{{c}_s}(\Psi_{{X}}(\C{F}_{\eta}),\Psi_{{X}}(\C{F}_{\eta})),
\]
which sends $u_{\eta}:{c}_{\eta l}^*(\C{F}_{\eta})\to {c}_{\eta r}^!(\C{F}_{\eta})$ to the composition
\[
{c}_{sl}^*(\Psi_{{X}}(\C{F}_{\eta}))\overset{BC}{\lra}\Psi_{C}({c}_{\eta l}^*(\C{F}_{\eta}))\overset{\Psi_{{C}}(u_{\eta})}{\lra} \Psi_{{C}}({c}_{\eta r}^!(\C{F}_{\eta}))\overset{BC}{\lra}
{c}_{sr}^!(\Psi_{{X}}(\C{F}_{\eta})).
\]
\end{Emp}

%Assume that ${c}$ extends to a correspondence ${c}':{C}'\to {X}'\times{X}'$
%over $\C{D}$ lifting some compactification $\ov{c}:\ov{C}\to\ov{X}\times\ov{X}$ of $c$.
\begin{Prop} \label{P:nearby}
In the situation \re{nearby}, the following diagram  is commutative
\begin{equation*} \label{Eq:nearby}
\CD
      \Hom_{{c}_{\eta}}(\C{F}_{\eta},\C{F}_{\eta})
  @>\C{Tr}_{{c}_{\eta}}>> H^0(Fix({c}_{\eta}), K_{\Fix({c}_{\eta})})\\
      @V\Psi_{{c}}VV            @V{\Sp_{\Fix({c})}}VV\\
     \Hom_{{c}_s}(\Psi_{{X}}(\C{F}_{\eta}),\Psi_{{X}}(\C{F}_{\eta}))
      @>\C{Tr}_{{c}_s}>>  H^0(Fix({c}_s), K_{Fix({c}_s)}).
\endCD
\end{equation*}
\end{Prop}

\begin{proof}
The assertion is a small modification \cite[Prop 1.3.5]{Va}, and the argument of \cite[Prop 1.3.5]{Va} proves our assertion as well. Alternatively, the assertion can be deduced from the general criterion of \cite[Section 4]{Va}. Namely, repeating the argument of \cite[4.1.4(b)]{Va} word-by-word, one shows that the nearby cycle functors $\Psi_{\cdot}$ together with base change morphisms define a compactifiable cohomological morphism in the sense of \cite[4.1.3]{Va}. Therefore the assertion follows from (a small modification of) \cite[Cor 4.3.2]{Va}.
\end{proof}

\begin{Lem} \label{L:nearby}
Let ${c}:{C}\to {X}\times {X}$ be a correspondence over $\C{D}$. Then for every $\C{F}\in  D_{ctf}^b({X},\La)$ and $u\in\Hom_{{c}}(\C{F},\C{F})$,  the following diagram is commutative
\begin{equation*} %\label{Eq:nearby}
\CD
         c_{sl}^*(\C{F}_s) @>u_s>> c_{sr}^!(\C{F}_s)\\
      @V\re{ULA}(a)VV            @VV\re{ULA}(a)V\\
     c_{sl}^*(\Psi_{{X}}(\C{F}_{\eta})) @>\Psi_{C}(u_{\eta})>> c_{sr}^!(\Psi_{{X}}(\C{F}_{\eta})).
\endCD
\end{equation*}

\end{Lem}
\begin{proof}
The assertion is a rather straightforward diagram chase. Indeed, it suffices to show the commutativity of the following diagram:
\begin{equation} \label{Eq:nearby2}
\CD
         c_{sl}^*(\C{F}_s) @= (c_{l}^*\C{F})_s @>u>> (c_{r}^!\C{F})_s @>BC>>     c_{sr}^!(\C{F}_s)\\
      @V\re{ULA}(a)VV            @V\re{ULA}(a)VV  @V\re{ULA}(a)VV            @VV\re{ULA}(a)V\\
     c_{sl}^*(\ov{\Psi}_{{X}}(\C{F}_{\ov{\eta}})) @>BC>> \ov{\Psi}_{C}((c_{l}^*\C{F})_{\ov{\eta}}) @>u>> \ov{\Psi}_{C}((c_{r}^!\C{F})_{\ov{\eta}}) @>BC>>   c_{sr}^!(\ov{\Psi}_{{X}}(\C{F}_{\ov{\eta}})).
\endCD
\end{equation}
We claim that all inner squares of \form{nearby2} are commutative. Namely, the middle inner square is commutative by functoriality, while the commutativity of the left and the right inner squares
follows by formulas $\ov{\Psi}_{\cdot}=i^*_s\circ i_{\ov{\eta}*}$ and definitions of the base change morphisms.
\end{proof}

Now we are ready to show \rp{const}.

\begin{Emp} \label{E:pfpconst}
{\bf Proof of \rp{const}.}
Without loss of generality we can assume that $s$ is a specialization of $t$ of codimension one.  Then there exists a spectrum of a discrete valuation ring $\C{D}$ and a morphism $f:\C{D}\to S$ whose image contains $s$ and $t$. Taking base change with respect to $f$ we can assume that $S=\C{D}$, $t=\ov{\eta}$ is the geometric generic point, while $s$ is the special point.

Then we have equalities
\[
\C{Tr}_{{c}_s}(u_s)=\C{Tr}_{c_s}(\Psi_{c}(u_{\eta}))=\Sp_{\Fix(c)}(\C{Tr}_{{c}_{\eta}}(u_{\eta}))= \ov{\Sp}_{\Fix(c)}(\pi_{\eta}^*(\C{Tr}_{{c}_{\eta}}(u_{\eta})))
=\ov{\Sp}_{\Fix(c)}(\C{Tr}_{{c}_{\ov{\eta}}}(u_{\ov{\eta}})),
\]
where

$\bullet$ the first equality follows from the fact that the isomorphism  $\C{F}_s\to\Psi_{X}(\C{F}_{\eta})$ from \re{ULA}(b) identifies $u_{s}$ with  $\Psi_{{c}}(u_{\eta})$ (by \rl{nearby});

$\bullet$ the second equality follows from the commutative diagram of \rp{nearby};

$\bullet$ the third equality follows from definition of $\Sp_X$ in \re{nearbyspec}(d);

$\bullet$ the last equality follows from the commutative diagram of \re{nearby}(a).

Now the assertion follows from \rl{toptriv}.
\end{Emp}

\end{document}